\documentclass{article}
\usepackage[utf8]{inputenc}
\usepackage{amsfonts, amsmath, amssymb, amsthm, authblk, graphicx, enumitem}
\usepackage{cite}
\usepackage[colorlinks=true,urlcolor=blue, citecolor=red,linkcolor=blue]{hyperref}
\usepackage[left=2cm, right=2cm, top=3cm]{geometry}
\usepackage[capitalise]{cleveref}
\newtheorem{theorem}{Theorem}
\newtheorem{acknowledgement}[theorem]{Acknowledgement}

\newtheorem{definition}[theorem]{Definition}

\newtheorem{lemma}[theorem]{Lemma}
\newtheorem{proposition}[theorem]{Proposition}
\newtheorem{remark}[theorem]{Remark}

\usepackage{dsfont}

\renewcommand{\d}{\displaystyle}
\usepackage{graphicx}
\usepackage{apptools}
\AtAppendix{\counterwithin{lem}{section}}
\newtheorem{lem}{Lemma}
\AtAppendix{\counterwithin{teoA}{section}}
\newtheorem{teoA}{Theorem}
\AtAppendix{\counterwithin{defA}{section}}
\newtheorem{defA}{Definition}
\newcommand{\pts}[1]{\left(#1\right)}  
                                  	  
\newcommand{\lvs}[1]{\left\{#1\right\}}

\newcommand{\C}{\mathbb{C}} 
\newcommand{\R}{\mathbb{R}}
\newcommand{\A}{\mathcal{A}}
\newcommand{\al}{\alpha}

\begin{document}
\title{Exact boundary controllability of a singular/degenerate wave equation via singular Sturm–Liouville theory
}

\author[1]{Marcos L\'opez-Garc\'ia \thanks{marcos.lopez@im.unam.mx}}
\affil[1]{Unidad Cuernavaca, Instituto de Matem\'aticas, Universidad Nacional Aut\'onoma de
M\'exico, M\'exico.}

\maketitle

\hfill\textit{In memory of my sister Chayo}
\begin{abstract}
We prove exact boundary controllability for a class of one-dimensional
singular/degenerate wave equations of Sturm--Liouville type,
\[
u_{tt}-(x^\alpha u_x)_x-\beta x^{\alpha-1}u_x-\mu x^{\alpha-2}u=0,
\qquad x\in(0,1),
\]
with a Dirichlet condition at the regular endpoint and a boundary control acting at the singular endpoint \(x=0\). The analysis is carried out in the fractional energy space %
\[
X=\mathcal H^{\nu+1/2}\times \mathcal H^{\nu-1/2},
\]
associated with the corresponding singular Sturm--Liouville operator.
Using the spectral decomposition induced by Bessel functions, we establish
admissibility of the boundary observation operator and derive precise lower estimates
for the observation coefficients.

Exact observability is obtained through Ingham-type inequalities together
with an abstract observability result for unitary groups. The proof treats
simultaneously the subcritical, critical logarithmic, and limit-point
regimes by means of singular Sturm--Liouville theory and boundary traces identified through the Lagrange bracket of the singular Sturm--Liouville expression.
\end{abstract}


\section{Introduction and main result}
The controllability of degenerate evolution equations has attracted
significant attention during the last decades.
For one-dimensional degenerate wave equations, a seminal contribution was
obtained by Gueye \cite{gueye}, who proved exact boundary controllability
for weakly degenerate hyperbolic equations by means of spectral methods,
nonharmonic Fourier series, and Bessel function expansions.
A remarkable feature of Gueye's result is that the control acts at the
degenerate endpoint itself.

More recently, Fragnelli, Mugnai, and Sbai
\cite{fragne}
investigated boundary controllability for degenerate/singular hyperbolic
equations with drift terms and singular potentials.
Their approach is based on energy methods and observability estimates.
However, the control is imposed at the nondegenerate and nonsingular
endpoint.

A similar situation appears in recent controllability results for
higher-order degenerate equations, where the control is typically applied
through the regular boundary; see, for instance,
\cite{akil}. This reflects the additional analytical difficulties arising when the
control acts precisely at the singular or degenerate endpoint.

In degenerate and singular problems, the formulation of boundary
controllability is strongly tied to the choice of the self-adjoint
realization of the underlying differential operator.
In particular, boundary terms arising from formal integrations by parts
may be misleading unless the singular endpoint, the associated boundary
form, and the admissible traces are identified at the operator-theoretic
level. For this reason, we first develop the singular Sturm--Liouville framework
before introducing the observation and control operators.

Within this framework, the boundary control operator is not imposed
\emph{ad hoc}; rather, it is naturally determined by the Lagrange bracket
associated with the singular Sturm--Liouville expression.

The present work follows a spectral and functional-analytic approach based
on singular Sturm--Liouville theory.
In contrast with the above-mentioned works, the control acts at the
degenerate/singular boundary point.
Moreover, our framework simultaneously covers subcritical, critical
(logarithmic), and limit-point regimes within a unified setting.
The analysis relies on self-adjoint realizations of singular operators,
weighted boundary traces arising from Lagrange brackets, and
nonharmonic Fourier techniques leading to Ingham-type observability
inequalities.\\

Let $T > 0$, and consider $\alpha,\beta,\mu\in \mathbb{R}$ with $0\leq\alpha<2$. In this work we analyze the exact controllability of the following hyperbolic system,
\begin{equation}\label{problem}
\left\{\begin{aligned}
u_{tt}-(x^\alpha u_x)_x-\beta x^{\alpha -1}u_x-\frac{\mu}{x^{2-\alpha}} u&=0, & & \text { if } (x,t)\in Q:= (0, 1) \times (0, T ), \\
\pts{B_\nu u}(t)= f(t),\quad u(1, t)&=0, & & \text { if } t \in (0, T), \\
u(x, 0) =u_{0}(x), \quad u_t(x, 0) &=u_{1}(x)& & \text { if } x\in (0, 1),
\end{aligned}\right.
\end{equation}
by means of a control of the form $f(t)\in L^2(0,T)$, where $U_0=(u_0,u_1)\in X$, and $B_\nu$ is a (weighted) boundary control operator at the singular endpoint
$x=0$:
\[
(B_\nu u)(t):=
\begin{cases}
\displaystyle
\lim_{x\to0^+}
\dfrac{u(x,t)}
{x^\sigma \log(1/x)},
& \nu=0,
\\[2.2ex]
\displaystyle
\lim_{x\to0^+}
x^{-\sigma+\sqrt{\Delta}/2}\,u(x,t),
& \nu>0.
\end{cases}
\]
where $\nu, \sigma, \Delta$ depend on $\alpha,\beta,\mu$ and are defined in (\ref{para}).\\

The simultaneous presence of degeneracy, drift, and an inverse-square
potential changes the singular structure at the controlled endpoint
\(x=0\). As a consequence, the boundary observation mechanism can no
longer be treated through standard weighted traces alone and requires
a genuinely singular Sturm--Liouville analysis.

The parameter $\nu$ determines the singular behavior of solutions near the
endpoint $x=0$, as well as the corresponding self-adjoint boundary
conditions and weighted boundary traces associated with the singular
Sturm--Liouville operator. For \(\nu=0\), \(B_\nu\) corresponds to a critical logarithmic boundary trace, whereas for \(\nu>0\) it is the weighted trace associated with the non-principal coefficient at the singular endpoint.


Our main result establishes exact boundary controllability from the
degenerate/singular endpoint within a unified framework covering the
subcritical, critical logarithmic, and limit-point regimes. The proof relies on singular Sturm--Liouville theory,
spectral decomposition techniques, and nonharmonic Fourier analysis.

\begin{theorem}\label{main}
Let $0\le \alpha<2$, $\beta,\mu\in\mathbb R$ with $(1-\alpha-\beta)^2\ge4\mu$. Then system \eqref{problem} is exactly controllable in any time
\[
T>\frac{4}{2-\alpha}.
\]
More precisely, for every $U_0,U_T\in X$
there exists a control
$
f\in L^2((0,T);\mathbb C)
$
such that the corresponding weak solution of the controlled system (\ref{problem}) satisfies
$
(u(T),u_t(T))=U_T.
$
\end{theorem}
A distinctive feature of the present work is that the observation
operator is not postulated a priori. Instead, it emerges naturally
from the singular Green identity and the asymptotic structure of
principal/non-principal solutions at the singular endpoint. This
provides a unified framework covering both limit-circle and
limit-point regimes and recovering the classical wave equation as
a particular case when $\alpha=\beta=\mu=0$.\\

The paper is organized as follows.
In Section \ref{2} we introduce the singular Sturm--Liouville framework, the
corresponding self-adjoint realizations, and the associated spectral
decomposition.
Section \ref{3} is devoted to the construction of the fractional spaces and to
the abstract wave dynamics generated by the singular operator.
Finally, Section \ref{4} establishes admissibility, exact observability, and
exact boundary controllability.


\section{Singular Sturm–Liouville framework}\label{2}
In this section we introduce the functional framework associated with the
singular Sturm--Liouville operator underlying equation~\eqref{problem}. We recall the basic notions of maximal and minimal operators,
self-adjoint realizations, and the Lagrange bracket. The notation and results from Sturm--Liouville theory used throughout this section can be found in~\cite{Zettl}.\\
The analysis developed below is also motivated by the fact that
the boundary observation structure appearing later in the control
problem is intrinsically determined by the singular behavior at
the endpoint \(x=0\).

\subsection{Functional setting}
First, for $\beta\in\mathbb{R}$ consider the weighted Lebesgue $L^2_{\beta}(0,1):=L^2((0,1);x^\beta\mathrm{d}x)$ with its inner product denoted by $\langle\cdot,\cdot\rangle_\beta$. Next, consider the differential expression $M$ defined by
\[Mu=-(pu_x)_x+qu\]
where $\d p(x) = x^{\al+\beta}, q(x) = -\mu x^{-2+\al+\beta}, w(x) = x^{\beta}$.\\

Observe that
\begin{equation*}\label{Acond2}
			1/p, q, w \in L_{\text{loc}}(0,1),\quad p,w >0\text{ on } (0,1),
\end{equation*} 
thus $Mu$ is defined a.e. for functions $u$ such that $u, pu_x\in AC_{\text{loc}}(0,1)$, where $AC_{\text{loc}}(0,1)$ is the space of all locally absolutely continuous functions in $(0,1)$.\\

For $\alpha,\beta,\mu\in \R$ consider the singular Sturm--Liouville operator $\mathcal{A}$ given by
\begin{equation}\label{A_ope}
\A u=\A_{\alpha,\beta,\mu} u:=w^{-1}Mu=-(x^\al u_x)_x-\beta x^{\al -1}u_x-\frac{\mu}{x^{2-\al}} u .
\end{equation}

Since $1/p,q,w\in L^1(1/2,1)$, the equation $\A u=\lambda u$ is regular at $x=1$ for any $\lambda\in \C$, see \cite[Definition 2.3.1, and (2.2.5)]{Zettl}. In this setting, we say that $x=1$ is a regular point.\\

We set
\begin{equation}\label{para}
\kappa_\alpha:=\frac{2-\alpha}{2},\,\,\Delta=\Delta(\alpha,\beta,\mu):=(1-\alpha-\beta)^2-4\mu,\,\,  \nu = \nu(\alpha,\beta,\mu):=\frac{\sqrt{\Delta}}{2\kappa_\alpha},\,\,\sigma=\sigma(\alpha,\beta):=\frac{1-\alpha-\beta}{2}.
\end{equation}

\noindent For $0\leq \al<2$, $\beta,\mu\in\R$ with \(\Delta\ge0\), we introduce the maximal domain,
\begin{equation*}\label{Dmax}
	D_{\max}:=\left\{u\in AC_{\text{loc}}(0,1)\, |\,pu_x\in AC_{\text{loc}}(0,1),\, u, \A u\in L^2_{\beta}(0,1)\right\}.
\end{equation*}
For $u,v \in D_{\max}$, the associated Lagrange bracket is defined by
\[
[u,v](x):=u(x)p(x)v'(x)-v(x)p(x)u'(x).
\]
The preminimal domain \(D_0 \) is defined by
\[
D_0=\{y\in D_{\max}: y \text{ has compact support in } (0,1)\}.
\]
The maximal and preminimal operators \(S_{\max}\) and \(S'_{\min}\), respectively, are defined by
\[
S_{\max}u=\A u, \qquad u\in D_{\max};
\qquad
S'_{\min}u=\A u, \qquad u\in D_0.
\]
The preminimal operator $S'_{\min}$ is densely defined and closable.
Its closure $(S_{\min},D_{\min})$ is called the minimal operator associated with $\mathcal{A}$.
Moreover, $S_{\min}$ is symmetric and
\[
S_{\min}^*=S_{\max},\quad \text{see \cite[Theorem 9.2.1]{Zettl}.}
\]
The analysis of the singular endpoint $x = 0$ depends on the parameter $\nu$
introduced in (\ref{para}), which determines the corresponding limit-circle or
limit-point classification.
\subsection{Self-adjoint realizations}
In this section we characterize the self-adjoint realizations of the
operator $\mathcal A$ at the singular endpoint $x=0$ by means of
singular Sturm--Liouville theory. We distinguish the analysis according to the value of $\nu$.
For $0 \le \nu < 1$ the singular endpoint $x = 0$ is limit-circle,
whereas for $\nu \ge 1$ it is limit-point. In the critical case $\nu= 0$, logarithmic singularities naturally arise
through the second linearly independent solution.\\

 \textbf{Subcritical case} $0<\nu<1$.\\
 In this case the functions
 \begin{equation}\label{prin}
\phi_+(x):=x^{\sigma+\sqrt{\Delta}/2},\qquad \phi_-(x):=\frac{1}{\sqrt{\Delta}}x^{\sigma-\sqrt{\Delta}/2},\quad x\in(0,1),
\end{equation}
are in $L^2_\beta(0,1)$, and satisfy $\A \varphi_{\pm}=0$. Indeed, $
\mathcal A(x^r)=0$ if and only if $
r_\pm
=
\sigma\pm\sqrt{\Delta}/2.$
 This implies that the endpoint $0$ is limit-circle (LC), see \cite[Definition 7.3.1, and Theorem 7.2.2]{Zettl}. Furthermore, the endpoint $0$ is nonoscillatory (NO), see \cite[Definition 7.3.1, and Theorem 7.3.1]{Zettl}. Then \cite[Theorem 10.5.1]{Zettl} implies that $S_{\min}$ is bounded below,  so we can consider the Friedrichs extension $S_F$ of $S_{\min}$, see Definition \ref{fried}.\\

From \cite[Definition 6.2.1]{Zettl} we have that $\phi_+$ is a principal solution at $0$, and $\phi_-$ is a non-principal solution at $0$. From \cite[Theorem 10.5.3, and Remark 10.5.1]{Zettl} we have
\[D(S_F)=\{u\in D_{\max}: [u,\phi_+](0)=u(1)=0\}.\]

On the other hand, we obtain that
    \begin{equation}\label{bounori}
    \left[u, \phi_{+}\right](0)=\lim_{x\rightarrow 0^+}[u, \phi_+](x)=\lim_{x\rightarrow 0^+}\lvs{\frac{u}{\phi_-}(x)[\phi_-,\phi_+](x)+[u,\phi_-](x)\frac{\phi_+}{\phi_-}(x)}=\lim_{x\rightarrow 0^+}\frac{u}{\phi_-}(x),
    \end{equation}
   because  $[\phi_-,\phi_+](0)=1$, $[u,\phi_-](0)$ is finite, see \cite[Lemma 10.2.3]{Zettl}, and $\lim_{x\rightarrow 0^+}\phi_+/\phi_-(x)=0$.\\
   
   The previous argument also yields a representation of the boundary form
in terms of weighted derivatives. More precisely,
   \begin{equation}\label{neumann}
    \left[u, \phi_{+}\right](0)=\lim_{x\rightarrow 0^+}\lvs{\frac{u'}{\phi'_-}(x)[\phi_-,\phi_+](x)+[u,\phi_-](x)\frac{\phi'_+}{\phi'_-}(x)}=\lim_{x\rightarrow 0^+}\frac{u'}{\phi'_-}(x).
    \end{equation}
   
   \textbf{Critical case} $\nu=0$.\\
   In this case the functions
   \begin{equation}\label{critical}
   y_{+}(x)=x^{\sigma}, \quad y_{-}(x)=-x^{\sigma}\ln x
   \end{equation}
   are in $L^2_\beta(0,1)$, and satisfy $\A y_{\pm}=0$. As in the last case, we can see that the endpoint $0$ is LCNO. Furthermore, $y_+$ is a principal solution at $0$, and $y_-$ is a non-principal solution at $0$. Therefore, the corresponding Friedrichs extension has the following domain
   \[D(S_F)=\{u\in D_{\max}: [u,y_+](0)=u(1)=0\}.\]
Since $[y_-,y_+](0)=1$, and proceeding as in (\ref{bounori}) with $y_{\pm}$ instead of $\phi_{\pm}$, we can see that
\begin{equation}\label{bouncri}
    \left[u, y_{+}\right](0)=\lim_{x\rightarrow 0^+}[u, y_+](x)=\lim_{x\rightarrow 0^+}\frac{u}{y_-}(x).
    \end{equation}
    
       The limit-point case, $\nu\ge 1$.
    In this case the function $\phi_-$ given in (\ref{prin}) is not in $L^2_{\beta}(0,1)$, then $x=0$ is limit-point (LP). Theorem 10.4.4 in \cite{Zettl} with $A_1=1, A_2=0$ implies that $D(\A)=\{u\in D_{\max}\,| u(1)=0\}$ is a self-adjoint domain.
		
		\begin{remark}
Assume that \(0<\nu<1\) and \(
\sigma-\sqrt{\Delta}/2\neq0
\)
(equivalently, \(\mu\neq0\)). Since
$
\phi_-'(x)
=
\Delta^{-1/2}(\sigma-\sqrt\Delta/2)
x^{\sigma-\sqrt\Delta/2-1},
$
the identity \eqref{neumann} yields
\[
[u,\phi_+](0)=0
\quad\text{if and only if}\quad
\lim_{x\to0^+}
x^{1-\sigma+\sqrt\Delta/2}u'(x)=0.
\]
Therefore, in this case the boundary
form may equivalently be expressed through a weighted Neumann trace. This equivalent representation is used only to interpret the boundary form;
the boundary control operator introduced later is chosen according to the
non-principal coefficient.
\end{remark}
		
\subsection{Spectral decomposition}
In this subsection we derive the spectral decomposition associated with
the self-adjoint realization of \(\A\) introduced above. The explicit
representation of the eigenfunctions in terms of Bessel functions allows
us to characterize the spectrum of \(\A\) and to obtain an orthonormal basis
of \(L^2_\beta(0,1)\).\\

The following proposition summarizes the results obtained in the previous subsections. In particular, we use (\ref{bounori}) and (\ref{bouncri}).
   \begin{proposition}\label{bases}
	Let $0\leq \al<2$, $\beta,\mu\in\R$, with $\Delta\ge 0$, and $\sigma, \nu$ defined in (\ref{para}). We set
	\[D(\mathcal{A}):=\left\{\begin{aligned}
	\{u\in D_{\max}\,|\lim_{x \to 0}\frac{u(x)}{x^{\sigma}\ln(1/x)}= 0, u(1)=0\} & &\text{if  } \nu=0,\\
	\{u\in D_{\max}\,| \lim_{x\rightarrow 0^+}x^{\sqrt{\Delta}/2-\sigma}u(x)=0,u(1)=0\} & &\text{if  } 0<\nu<1,\\
		\{u\in D_{\max}\,| u(1)=0\}& & \text{if  } \nu\ge1.\end{aligned}\right.
\]
	In the limit-circle case $0\le \nu<1$, the above realization
coincides with the Friedrichs extension associated with the singular
endpoint $x=0$. In the limit-point case $\nu\ge1$, the self-adjoint
realization is uniquely determined by the condition $u(1)=0$.\\
Then $\mathcal{A}:D(\mathcal{A})\subset L^2_{\beta}(0,1)\rightarrow L^2_{\beta}(0,1)$ is a self-adjoint operator for each $\nu=\nu(\alpha,\beta,\mu)$. Furthermore, for $\nu$ fixed, the family 
	\begin{equation}\label{Phik}
		\Phi_k(x)=\Phi_{k,\nu}(x):=\frac{(2\kappa_\alpha)^{1/2}}{|J'_\nu(j_{\nu,k})|}x^{\sigma}J_\nu(j_{\nu,k}x^{\kappa_\alpha}),\quad k\geq 1,
	\end{equation}
	is an orthonormal basis for $L^2_\beta(0,1)$ such that $\{\Phi_k\}_{k\geq 1}\subset D(\mathcal{A})$, and 
	\begin{equation}\label{lambdak}
		\mathcal{A}\Phi_k=\lambda_k \Phi_k, \quad \lambda_k=\kappa_\alpha^2 (j_{\nu,k})^2,\quad k\geq 1.
	\end{equation}
	\end{proposition}

	\begin{proof} The self-adjointness of the stated realization follows from the
previous subsection. It remains to prove the spectral assertion.\\
	Assume that $u$ is smooth on $(0,1)$ and set $u(x)=x^\sigma v\!\left(x^{\kappa_\alpha}\right)$, then  
\[
\A u=\lambda u \quad \text{iff} \quad z^2 v''(z)+ z v'(z)+\left(\frac{\lambda}{\kappa_\alpha^2}z^2-\nu^2\right)v(z)=0.
\]
Therefore, for $\lambda>0$, the general solution of $\A u=\lambda u$ is
\[
u(x)=x^\sigma\Bigl(
c_1 J_\nu\!\bigl(\tfrac{\sqrt{\lambda}}{\kappa_\alpha}x^{\kappa_\alpha}\bigr)
+
c_2 Y_\nu\!\bigl(\tfrac{\sqrt{\lambda}}{\kappa_\alpha}x^{\kappa_\alpha}\bigr)
\Bigr),
\]
where \(J_\nu\) and \(Y_\nu\) denote the Bessel functions
of the first and second kind, respectively.\\

\emph{The limit-circle case, $0\le \nu <1$}.\\
From (\ref{asinu}) and (\ref{asin0}) we obtain that
\[\lim_{x\rightarrow 0^+}\frac{Y_0\!\bigl(\tfrac{\sqrt{\lambda}}{\kappa_\alpha}x^{\kappa_\alpha}\bigr)}{\ln(1/x)}=-\frac{2\kappa_\alpha}{\pi},\quad\text{and}\quad\lim_{x\rightarrow 0^+}x^{\sqrt{\Delta}/2}Y_\nu\!\bigl(\tfrac{\sqrt{\lambda}}{\kappa_\alpha}x^{\kappa_\alpha}\bigr)=-\frac{\Gamma(\nu)}{\pi}\left(\frac{2\kappa_\alpha}{\sqrt{\lambda}}\right)^\nu\]
for $0<\nu<1$, and the boundary condition at \(x=0\) forces \(c_2=0\). \\

\noindent\emph{The limit-point case, $\nu \geq 1$}.\\
From (\ref{asinu}) and (\ref{asina}) it follows that $x^{\sigma}Y_\nu\!\bigl(\tfrac{\sqrt{\lambda}}{\kappa_\alpha}x^{\kappa_\alpha}\bigr)\notin L^2_\beta(0,1)$, then $c_2=0$.\\

The condition $u(1)=0$ implies that $\sqrt{\lambda}/\kappa_\alpha$ is a zero of the Bessel function $J_\nu$. Correspondingly, the eigenfunctions are
\[
\widetilde{\Phi}_k(x)=x^\sigma J_\nu(j_{\nu,k}x^{\kappa_\alpha}),\qquad k\ge1.
\]
The change of variable $y=x^{\kappa_\alpha}$, and the standard orthogonality formula
\[
\int_{0}^{1} x\, J_{\nu}(j_{\nu,m}x)\, J_{\nu}(j_{\nu,n}x)\, dx
=
\frac{ |J_{\nu+1}(j_{\nu,n})|^{2}}{2}\,\delta_{mn}.
\]
implies that $\widetilde{\Phi}_k$, $k\geq 1$, is an orthogonal family in $L^2_\beta(0,1)$. Moreover,
\[\|\widetilde{\Phi}_k\|_{L^2_\beta(0,1)}^2=\frac1{\kappa_\alpha}\int_0^1 y\,|J_\nu(j_{\nu,k}y)|^2\,dy=\frac1{2\kappa_\alpha}\,|J'_\nu(j_{\nu,k})|^2\]
because $J_\nu(j_{\nu,k})=0$ and $J'_\nu(j_{\nu,k})=-J_{\nu+1}(j_{\nu,k})$, see (\ref{recur}).\\

Therefore, $\Phi_k\in D(\A)$ for all $k\ge 1$, $\nu=\nu(\alpha,\beta,\mu)$.\\

For $\nu>-1$, Hochstadt proved in \cite{Hoch} that the family
	\[\Theta_k(x):=\frac{2^{1/2}}{|J'_\nu(j_{\nu,k})|}x^{1/2}J_\nu(j_{\nu,k}x),\quad k\geq 1,\]
	is an orthonormal basis for $L^2(0,1)$.\\ 
	
	Let $\mathcal{U}$ be the unitary operator $\mathcal{U}:L^2(0,1)\rightarrow L^2_\beta(0,1)$ given by
	\begin{equation*}
		\mathcal{U}u(x):=\kappa_\alpha^{1/2}x^{-\alpha/4-\beta/2}u(x^{\kappa_\alpha}), \quad u\in L^2(0,1).
	\end{equation*}
	Notice that $\mathcal{U}\Theta_k=\Phi_k$, $k\geq 1$, therefore $\Phi_k$, $k\geq 1$, $\nu\ge0$, is an orthonormal basis for $L^2_\beta(0,1)$.
\end{proof}


\section{Fractional spaces and wave dynamics}\label{3}
In this section we introduce the scale of fractional spaces associated
with the singular Sturm--Liouville operator \(\A\), and study the
corresponding abstract wave dynamics.
The spectral decomposition obtained in Section \ref{2} naturally induces a
Hilbert scale adapted to the degenerate/singular structure of the
equation.
Within this framework, the wave equation can be formulated as a
first-order evolution system generated by a skew-adjoint operator on the
energy space.
We also derive the spectral representation of solutions, which will play
a fundamental role in the controllability analysis developed in Section \ref{4}.\\
The abstract setting adopted in this section follows the semigroup
framework for second-order evolution equations developed in
Tucsnak and Weiss \cite{tucwei}. In particular, the fractional scales
associated with the operator \(\A\), as well as the state,
control, and observation spaces introduced later, are formulated within
this operator-theoretic framework.
\subsection{Fractional spaces}
We now introduce a family of interpolation spaces for the initial data. For any $s\geq 0$, we define
\[\mathcal{H}^{s}=\mathcal{H}^{s}(0,1):=\left\{u=\sum_{k=1}^\infty a_{k} \Phi_{k}:\|u\|_{\mathcal{H}^{s}}^{2}=\sum_{k=1}^\infty \lambda_{k}^{s}|a_{k}|^{2} <\infty\right\},\]
and we also consider the corresponding dual spaces
\[\mathcal{H}^{-s}:=\left[\mathcal{H}^{s}(0,1)\right]^{\prime}.\]
It is well known that $\mathcal{H}^{-s}$ is the dual space of $\mathcal{H}^{s}$ with respect to the pivot space $L^2_\beta(0,1)$, i.e
\[\mathcal{H}^s\hookrightarrow \mathcal{H}^0=L^2_{\beta}(0,1)=\left(L^2_{\beta}(0,1)\right)'\hookrightarrow \mathcal{H}^{-s},\quad s>0. \]
For $s>0$, equivalently, $\mathcal{H}^{-s}$ is the completion of $L^2_\beta(0,1)$ with respect to the norm
\[\|u\|^2_{-s}:=\sum_{k=1}^{\infty}\lambda_k^{-s}|\langle u,\Phi_k\rangle_\beta|^2,\quad u\in L^2_{\beta}(0,1).\]

Notice that $D(\A)=\mathcal{H}^{2}$, and we also have
\[\langle \A u,u\rangle_\beta=\sum_{k=1}^\infty\lambda_k |a_k|^2\ge \kappa_\alpha^2(j_{\nu,1})^2\|u\|_{\beta}^2\]
for all $u=\sum_{k=1}^\infty a_{k} \Phi_{k}\in D(\A)$. Therefore, $(\A,D(\A))$ is a strictly positive operator. \\

From now on, we set $H:= L^2_\beta(0,1)$, and
\[H_{1/2}:=D(\mathcal{A}^{1/2}).\]
From \cite[Remark 3.4.4]{tucwei} we have that $H_{1/2}$ is the completion of $D(\mathcal{A})$ with respect to the norm
$$
\|u\|_{1/2}:=\sqrt{\left\langle \mathcal{A} u, u\right\rangle_\beta}$$
for all $u\in D(\mathcal{A})$. From \cite[Proposition 3.4.8]{tucwei} it follows that $H_{1/2}=\mathcal{H}^1$.\\

\begin{remark}\label{baes}Notice that $\lambda_k^{-s/2}\Phi_k$, $k\ge 1$, is an orthonormal basis for $\mathcal{H}^s$, $s\in\R$. In particular, if $u\in \mathcal{H}^s$ then
\begin{equation}\label{expan}
u
=
\sum_{k=1}^{\infty}
\langle u,\Phi_k\rangle_{\mathcal{H}^s}\,
\lambda_k^{-s}\,\Phi_k
\quad\text{and}\quad
\|u\|_{\mathcal{H}^s}^2
=
\sum_{k=1}^{\infty}
\big|\langle u,\Phi_k\rangle_{\mathcal{H}^s}\big|^2\,
\lambda_k^{-s}.
\end{equation}
In particular, $\langle u,\Phi_k\rangle_{\mathcal H^{s+\ell}}=(\lambda_k)^{\ell}\langle u,\Phi_k\rangle_{\mathcal H^{s}}$ for all $u\in\mathcal H_{s+\ell}$, $\ell\ge 1$.
\end{remark}
\subsection{The abstract wave equation}
The spectral decomposition obtained above naturally induces a scale of
fractional spaces adapted to the singular wave equation.\\
Define $\mathcal{X}_{0}:=H_{1/2} \times H=\mathcal{H}^1\times \mathcal{H}^0$, with the scalar product
$$
\left\langle\left(\begin{array}{l}
w_1 \\
v_1
\end{array}\right),\left(\begin{array}{l}
w_2 \\
v_2
\end{array}\right)\right\rangle_{\mathcal{X}_{0}}
=\left\langle w_1, w_2\right\rangle_{\mathcal{H}^1}+\left\langle v_1, v_2\right\rangle_{\mathcal{H}^0} .
$$
Consider the dense subspace ${D}(\mathbf{A}):=D\left(\mathcal{A}\right) \times D(\mathcal{A}^{1/2})$ of $\mathcal{X}_{0}$, see \cite[Proposition 3.4.3]{tucwei}, and the linear operator $\mathbf{A}: D(\mathbf{A}) \rightarrow \mathcal{X}_{0}$ given by
$$
\mathbf{A}=\left[\begin{array}{cc}
0 & I \\
-\mathcal{A} & 0
\end{array}\right], \quad \text { i.e. } \quad \mathbf{A}\left(\begin{array}{l}
\varphi \\
\psi
\end{array}\right)=\left(\begin{array}{c}
\psi \\
-\mathcal{A}\varphi
\end{array}\right) .
$$
Then $\mathbf{A}$ is skew-adjoint on $\mathcal{X}_{0}$ and $0 \in \rho(\mathbf{A})$, see \cite[Proposition 3.7.6]{tucwei}.\\ 

For all $k\in \mathbb{N}$ we define 
$$\gamma_k:=\lambda^{1/2}_k=\kappa_\alpha j_{\nu,k},\quad \Phi_{-k}:=-\Phi_k,\quad \gamma_{-k}:=-\gamma_k.$$
Proposition \ref{bases} and \cite[Proposition 3.7.7]{tucwei} imply that 
 $\mathbf{A}$ is diagonalizable, with the eigenvalues $i \gamma_k$ corresponding to the orthonormal basis of eigenvectors
$$
\phi_k:=\frac{1}{\sqrt{2}}\left(\begin{array}{c}
\frac{1}{i \gamma_k} \Phi_k \\
\Phi_k
\end{array}\right) \quad \text{for all   }k \in \mathbb{Z}^*:=\mathbb{Z}\backslash\{0\} .
$$

From \cite[Proposition 3.8.7]{tucwei} we have that $\mathbf{A}$ generates a unitary group $\mathbb{T}(t)$, $t\in\mathbb{R}$, on $\mathcal{X}_{0}$.\\

For any $s\in\mathbb{R}$, we define the corresponding scale of state spaces
\[
\mathcal{X}_{s} := \mathcal{H}^{s+1} \times \mathcal{H}^{s},
\quad
\|(w_0,w_1)\|_{\mathcal{X}_{s}}^2
:=
\|w_0\|_{\mathcal{H}^{s+1}}^2
+
\|w_1\|_{\mathcal{H}^{s}}^2.
\]

The operator $\mathbf{A}$ is skew-adjoint on $\mathcal{X}_{s}$ with domain $D(\mathbf{A})=\mathcal{X}_{s+1}$,
and therefore generates a unitary $C^0$-group $\mathbb{T}(t)$, $t\in\mathbb{R}$, on $\mathcal{X}_{s}$.

\begin{remark}[Triple Gelfand] \label{gelfand}
The pivot space for the abstract control framework will be chosen as
\[
X:=(\mathcal{X}_{s},\|\cdot\|_{\mathcal{X}_{s}}),\quad\text{with   }s=\nu-1/2.
\]
We write $X_1:=D(\mathbf{A})=\mathcal{X}_{s+1}$ with the norm
\[\|(w_0,w_1)\|_{X_1}^2:=\|\mathbf{A}(w_0,w_1)\|_{X}^2=\|(w_1,-\mathcal Aw_0)\|_{\mathcal X_s}^2=\|(w_0,w_1)\|_{\mathcal{X}_{s+1}}^2\]
for all $(w_0,w_1)\in X_1$, thus $(X_1,\|\cdot\|_{X_1})=(\mathcal{X}_{s+1},\|\cdot\|_{\mathcal{X}_{s+1}})$. Furthermore, \(X_{-1}\) denotes the extrapolation space obtained as the
completion of \(X\) with respect to the norm
$
\|Z\|_{X_{-1}}:=\|\mathbf A^{-1}Z\|_{X}
$. With the above spectral scale, \(X_{-1}\) can be identified with
\(\mathcal X_{s-1}\). Moreover, \(X_{-1}\) is the dual of
\(X_1\) with respect to \(X\) see \cite[Proposition 2.10.2]{tucwei}.
\end{remark}
%
\subsection{Spectral representation of solutions}
By Remark \ref{baes} we have that $\lambda_{|k|}^{-s/2}\phi_k$, $k\in\mathbb{Z}^*$, forms an orthonormal basis for $\mathcal{X}_{s}$. Assume that $(w_0,w_1)\in \mathcal{X}_{s}$ 
then
\[
\binom{w_0}{w_1}
=
\sum_{k\in\mathbb{Z}^*}
\left\langle
\binom{w_0}{w_1},
\lambda_{|k|}^{-s/2}\,\phi_k
\right\rangle_{\mathcal{X}_s}
\lambda_{|k|}^{-s/2}\,\phi_k.
\]
By direct computation and the definition of the eigenvectors $\phi_k$ we have
\begin{eqnarray}\label{group}
\mathbb{T}(t)
\binom{w_0}{w_1}
&=&
\sum_{k\in\mathbb{Z}^*}
\left\langle
\binom{w_0}{w_1},
\lambda_{|k|}^{-s/2}\,\phi_k
\right\rangle_{\mathcal{X}_s}
\lambda_{|k|}^{-s/2}\,\phi_k\,
e^{i\gamma_k t}\\
&=&
\frac12
\sum_{k\in\mathbb{Z}^*}
\left(
\frac{i}{\gamma_k}
\langle w_0,\Phi_k\rangle_{\mathcal{H}^{s+1}}
+
\langle w_1,\Phi_k\rangle_{\mathcal{H}^{s}}
\right)
\lambda_{|k|}^{-s}
\binom{\frac{1}{i\gamma_k}\,\Phi_k}{\Phi_k}
\,e^{i\gamma_k t}. \notag
\end{eqnarray}

\paragraph{Notation.}
Given $(w_0,w_1)\in \mathcal{X}_s$, we denote by
\[
\mathbb {T}(t)(w_0,w_1) = \big(w(\cdot,t),\,w_t(\cdot,t)\big)
\]
the solution of the homogeneous system $\dot W=\mathbf{A}W$, $W(0)=(w_0,w_1)$, $0\le t\le T$.
In what follows, we shall identify $w(t)$ with the first component
$\pi_1\big(\mathbb{T}(t)(w_0,w_1)\big)=w(\cdot,t)$ whenever no confusion arises.\\

Therefore (\ref{group}) implies that 
\begin{equation}\label{repre}
w(x, t)=\sum_{k=1}^\infty  \left[b_k e^{i \gamma_k t}+b_{-k} e^{-i \gamma_k t}\right]\Phi_k(x)
\end{equation}
where
\begin{equation}\label{coef}
\begin{aligned}
b_k &:=
\frac12
\left(
\langle w_0,\Phi_k\rangle_{\mathcal H^{s+1}}
\lambda_k^{-s-1}
-
i\,
\langle w_1,\Phi_k\rangle_{\mathcal H^{s}}
\lambda_k^{-1/2-s}
\right),
\\[1ex]
b_{-k} &:=
\frac12
\left(
\langle w_0,\Phi_k\rangle_{\mathcal H^{s+1}}
\lambda_k^{-s-1}
+
i\,
\langle w_1,\Phi_k\rangle_{\mathcal H^{s}}
\lambda_k^{-1/2-s}
\right).
\end{aligned}
\end{equation}

By applying \cite[Proposition 3.8.7]{tucwei} we obtain a result on the existence and uniqueness of the following homogeneous system,
\begin{equation}\label{homo}
\left\{\begin{aligned}
w_{tt}-(x^\alpha w_x)_x-\beta x^{\alpha -1}w_x-\frac{\mu}{x^{2-\alpha}} w&=0, & & \text { if } (x,t)\in Q:= (0, 1) \times (0, T ), \\
w(x, 0) =w_{0}(x), \quad w_t(x, 0) &=w_{1}(x)& & \text { if } x\in (0, 1).
\end{aligned}\right.
\end{equation}
The boundary conditions are those encoded in the self-adjoint realization
of \(\mathcal A\). 
\begin{proposition} \label{unique} Let $T>0$, $0\leq \alpha < 2$, $\beta, \mu, s\in \mathbb{R}$ with $\Delta \ge 0$.
 For each $\left(w_0, w_1\right) \in \mathcal{X}_{s}$ there exists a unique mild solution $w \in C^0([0, T] ; \mathcal{H}^{s+1}) \cap C^1([0, T] ;\mathcal{H}^{s})$ of the system (\ref{homo}) that satisfies
$$
\left\|w(t)\|_{\mathcal{H}^{s+1}}^2+\right\| w_t (t)\|^2_{\mathcal{H}^{s}} = \left\|w_0\|_{\mathcal{H}^{s+1}}^2+\right\| w_1\|^2_{\mathcal{H}^{s}},\quad\text{for all }0\leq t\leq T.$$

Moreover, if $\left(w_0, w_1\right) \in \mathcal{X}_{s+1}$, there exists a unique solution $w \in C^0([0, T];\mathcal{H}^{s+2}) \cap C^1([0, T] ; \mathcal{H}^{s+1})\cap  C^2([0,T];\mathcal{H}^s)$ of the system (\ref{homo}).
\end{proposition}
\section{Exact controllability}\label{4}
In this section we formulate the boundary control problem associated with
system \eqref{problem} within the abstract semigroup framework introduced in the
previous sections. The singular structure of the endpoint $x=0$ naturally
determines the observation operator through the corresponding boundary traces
and the Lagrange bracket associated with the Sturm--Liouville expression.\\
We first introduce the observation and control operators on the fractional
energy space $
X=\mathcal H^{\nu+1/2}\times \mathcal H^{\nu-1/2},
$ and establish the admissibility of the observation operator for the unitary
group generated by the wave operator. Using the spectral representation of solutions and Ingham-type inequalities,
we establish admissibility, exact observability, and exact boundary
controllability within this abstract framework.
\subsection{Observation and control operators}
From now on we consider the Gelfand triple $(X_1,X,X_{-1})$ given in Remark \ref{gelfand} with $s=\nu-1/2$.\\
The definition of the observation operator below is motivated by the
singular Lagrange boundary form associated with the Sturm--Liouville
expression. A detailed derivation of the corresponding boundary traces
and their relation with the transposition formulation is given later in
Section \ref{trans}.\\

For $0\le \alpha<2$, $\beta,\mu\in\mathbb R$ we define the abstract observation operator $
\mathbb B^*_\nu=\mathbb B_{\nu(\alpha,\beta,\mu)}^*:X_1\to\mathbb C
$ by
\begin{equation}\label{obsop}
\mathbb B_\nu^*
\binom{w_0}{w_1}
:=
\begin{cases}
-[w_0,y_-](0)
=
\mathcal O_{-\sigma}(w_0),
& \nu=0,
\\[2mm]
-\sqrt\Delta\,[w_0,\phi_-](0)
=
\sqrt\Delta\,
\mathcal O_{-\sigma-\sqrt\Delta/2}(w_0),
& \nu>0.
\end{cases}
\end{equation}
where $ \phi_-,y_-$ are given in (\ref{prin}) and (\ref{critical}), and the boundary trace functionals $\mathcal O_\delta$ are defined by
\[
\mathcal O_\delta(u)
:=
\lim_{x\to0^+}x^\delta u(x),
\]
whenever the corresponding limit exists.
\begin{proposition}\label{admi} Let $0\le \alpha<2$, $\beta,\mu\in \mathbb R$ with $\Delta\ge 0$. Then $\mathbb{B}^*_\nu\in \mathcal L (X_1,\mathbb C)$.
\end{proposition}
\begin{proof} For any $\binom{w_0}{w_1}\in X_1$ we can write
\begin{equation}\label{expadmi}
w_0
=
\sum_{k=1}^{\infty}
\langle w_0,\Phi_k\rangle_{\mathcal{H}^{\nu+3/2}}\,
\lambda_k^{-\nu-3/2}\,\Phi_k,\quad \text{see (\ref{expan})}.
\end{equation}

\textit{Case} $\nu=0$.  By Lemmas \ref{reduce}-\ref{boundtrac},  (\ref{lambdak}), (\ref{expan}), (\ref{obsop}-\ref{expadmi}), and (\ref{below}) we have that there exists a constant $C=C_{\alpha,\beta,\mu}>0$ such that
\[
\begin{aligned}
\left|
\mathbb B_0^*
\begin{pmatrix}
w_0\\ w_1
\end{pmatrix}
\right|
&\le
C
\sum_{k=1}^\infty
|\langle w_0,\Phi_k\rangle_{\mathcal H^{3/2}}|
\lambda_k^{-5/4}\\
&\le
C
\left(
\sum_{k=1}^\infty
|\langle w_0,\Phi_k\rangle_{\mathcal H^{3/2}}|^2
\lambda_k^{-3/2}
\right)^{1/2}
\left(
\sum_{k=1}^\infty
\lambda_k^{-1}
\right)^{1/2}\\
&\le
C
\|w_0\|_{\mathcal H^{3/2}}
\le
C
\left\|(w_0,w_1)
\right\|_{ X_1}.
\end{aligned}
\]


\textit{Case} $\nu>0$. By Lemmas \ref{reduce}-\ref{boundtrac},  (\ref{lambdak}), (\ref{expan}), (\ref{obsop}-\ref{expadmi}), and (\ref{below}) we have that there exists a constant $C=C_{\alpha,\beta,\mu}>0$ such that
\[\left|\mathbb{B}_\nu^*\binom{w_0}{w_1} \right|\leq
 \frac{C_{\alpha,\beta,\mu}(2\kappa_\alpha)^{1/2}}{2^{\nu}\,\Gamma(\nu+1)}\sum_{k=1}^{\infty} \frac{\left(j_{\nu, k}\right)^\nu}{\left|J_\nu^{\prime}\left(j_{\nu, k}\right)\right|}|\langle w_0,\Phi_k\rangle_{\mathcal{H}^{\nu+3/2}}|\,
\lambda_k^{-\nu-3/2}\leq C_{\alpha,\beta,\mu} \|(w_0,w_1)\|_{X_1}.
\]
Combining both cases, we conclude that
$
\mathbb B_\nu^*\in\mathcal L(X_1,\mathbb C).
$
\end{proof}

The corresponding control operator $\mathbb{B}_\nu \in \mathcal L(\mathbb C, X_{-1})$
is defined by transposition as
\begin{equation}\label{conope}
\langle \mathbb{B}_\nu f,(w_0,w_1)\rangle_{X_{-1},X_1}= f \, \mathbb{B}^*_\nu\binom{w_0}{w_1},\quad f\in \mathbb C, \binom{w_0}{w_1}\in X_1.
\end{equation}

\begin{lemma}[Admissibility]\label{admiest}
Let $T>0$, $0\le \alpha<2$, $\beta,\mu\in\mathbb R$ with $\Delta\ge0$.
Then there exists $C_T>0$ such that
\[
\int_0^T
\left|
\mathbb B_\nu^*\mathbb T(t)(w_0,w_1)
\right|^2\,dt
\le
C_T\|(w_0,w_1)\|_X^2,
\qquad \forall
(w_0,w_1)\in X_1.
\]
\end{lemma}
\begin{proof}
Let $(w_0,w_1)\in X_1$, then Proposition \ref{unique} implies that $\mathbb T(t)\binom{w_0}{w_1}=(w(t),w_t(t))\in X_1$ for $t\in [0,T]$. We use Lemma \ref{boundtrac}, the representation formula (\ref{repre}) with $s=\nu+1/2$, and the upper estimate in Ingham's inequality for uniformly separated
frequencies. Our sequence $(\gamma_k)_{k\in \mathbb{Z}}$ satisfies $\delta:=\inf_{k\neq 0}(\gamma_{k+1}-\gamma_k)>0$. \\
Assume that $\nu>0$, arguing as in the proof of Proposition \ref{admi}, and by Remark \ref{baes}, we obtain
 \begin{eqnarray*}
\int_0^T \left|\mathbb{B}^*_\nu\mathbb T(t)\binom{w_0}{w_1}\right|^2\,\mathrm dt&\le&C_T\sum_{k=1}^\infty \frac{\left(j_{\nu, k}\right)^{2 \nu}}{\left[J_\nu^{\prime}\left(j_{\nu, k}\right)\right]^2}\left(\big|\langle w_0,\Phi_k\rangle_{\mathcal H^{\nu+3/2}}\big|^{2}\,\lambda_k^{-2\nu-3}
\;+\;
\big|\langle w_1,\Phi_k\rangle_{\mathcal H^{\nu+1/2}}\big|^{2}\,\lambda_k^{-2\nu-2}
\right)\\
										&=& C_T\sum_{k=1}^\infty\frac{\left(j_{\nu, k}\right)^{2 \nu}}{\left[J_\nu^{\prime}\left(j_{\nu, k}\right)\right]^2}\left(\big|\langle w_0,\Phi_k\rangle_{\mathcal H^{\nu+1/2}}\big|^{2}\,\lambda_k^{-2\nu-1}
\;+\;
\big|\langle w_1,\Phi_k\rangle_{\mathcal H^{\nu-1/2}}\big|^{2}\,\lambda_k^{-2\nu}
\right)\\
											     &\le& C_T\sum_{k=1}^\infty \left(\big|\langle w_0,\Phi_k\rangle_{\mathcal H^{\nu+1/2}}\big|^{2}\lambda_k^{-\nu-1/2}+\big|\langle w_1,\Phi_k\rangle_{\mathcal H^{\nu-1/2}}\big|^{2}\lambda_k^{-\nu+1/2}\right)=C_T\,\|(w_0,w_1)\|_X^2.
\end{eqnarray*}
The case \(\nu=0\) follows in the same way by using the corresponding
trace formula in Lemma \ref{boundtrac}.
\end{proof}
\begin{remark}\label{admiobco}
\cite[Definition 4.3.1 and (4.3.3)]{tucwei} imply that $\mathbb B_\nu^* $ is an admissible observation operator for $\mathbb{T}(t)$, $t\in \R$.
\end{remark}
\begin{proposition}[Exact observability]\label{observa}
Let $0\le \alpha<2$, $\beta,\mu\in\mathbb R$ with $\Delta\ge0$.
If $T>4/(2-\alpha)$, then there exists $c_T>0$ such that
\[
c_T\|(w_0,w_1)\|_X^2
\le
\int_0^T
\left|
\mathbb B_\nu^*\mathbb T(t)(w_0,w_1)
\right|^2\,dt,
\qquad \forall
(w_0,w_1)\in X_1.
\]
\end{proposition}
\begin{proof}
First, we compute the upper Beurling density $D^+$ of $(\gamma_k)_{k\in \mathbb{Z}}$ defined in (\ref{dens}). By definition of $k^+(r)$ we have
$
\gamma_{k^+(r)} - \gamma_1
\sim r
\quad \text{as } r\to+\infty,
$
hence
\[
\frac{r}{k^+(r)}
\sim
\frac{1}{k^+(r)}
\sum_{k=1}^{k^+(r)-1}
(\gamma_{k+1}-\gamma_k).
\]
Lemma \ref{consec} implies that
\[
\frac1n\sum_{k=1}^{n-1}(\gamma_{k+1}-\gamma_k)\to \kappa_\alpha\pi,
\]
hence
\[
\frac{r}{k^+(r)}
\to
\kappa_\alpha \pi \quad \text{as } r\to+\infty, \quad\text{thus    } D^+=\limsup_{r\to+\infty}\frac{k^+(r)}{r}=\frac{1}{\kappa_\alpha \pi}.\]

The condition \(T>4/(2-\alpha)\) implies $T>2\pi D^+,$ hence Theorem \ref{beur} applied on the interval \(I=(0,T)\) gives
\[
c_T\sum_{k\in\mathbb Z^*}|a_k|^2
\le
\int_0^T
\left|
\sum_{k\in\mathbb Z^*}a_k e^{i\gamma_k t}
\right|^2dt
\]
for every finitely supported sequence \((a_k)_{k\in\mathbb Z^*}\).\\

Now let \((w_0,w_1)\in X_1\). From the representation formula \eqref{repre} with $s=\nu+1/2$ and
the definition of \(\mathbb B_\nu^*\), we can write
\[
\mathbb B_\nu^*\mathbb T(t)(w_0,w_1)
=
\sum_{k\ge1}
d_k
\left(
b_k e^{i\gamma_k t}+b_{-k}e^{-i\gamma_k t}
\right),
\]
where \(d_k=\mathbb B_\nu^*(\Phi_k,0)\). In the two cases
\(\nu=0\), \(\nu>0\), Lemmas \ref{reduce}-\ref{boundtrac} imply that there exists a constant \(c>0\) such that
\[
|d_k|^2\ge c\,(\kappa_\alpha j_{\nu,k})^{2\nu+1}=c\,\lambda_k^{\nu+1/2},
\qquad k\ge 1.
\]
Applying Theorem \ref{beur} to the sequence $a_k=d_k b_k$, $a_{-k}=d_k b_{-k}$, we obtain
\[
\begin{aligned}
\int_0^T
\left|
\mathbb B_\nu^*\mathbb T(t)(w_0,w_1)
\right|^2dt
&\ge
c_T
\sum_{k\ge1}
|d_k|^2\left(|b_k|^2+|b_{-k}|^2\right)  
\ge
c_T
\sum_{k\ge1}
\lambda_k^{\nu+1/2}
\left(|b_k|^2+|b_{-k}|^2\right)\\
&=c_T
\sum_{k\ge1}
\left(
|\langle w_0,\Phi_k\rangle_{\mathcal H^{\nu+3/2}}|^2
\lambda_k^{-\nu-5/2}
+
|\langle w_1,\Phi_k\rangle_{\mathcal H^{\nu+1/2}}|^2
\lambda_k^{-\nu-3/2}
\right)\\
&=
c_T\|(w_0,w_1)\|_X^2 .
\end{aligned}
\]
\end{proof}
\subsection{Definition of weak solution}\label{trans}
In this subsection we motivate the trace identities appearing in \eqref{obsop} and explain the definition of the observation operator
$\mathbb B_\nu^*$ by analyzing the boundary contribution arising from Green's identity at the singular endpoint $x=0$. Then we introduce the definition of weak solution $u$ for the system (\ref{problem}).\\
Assume that $u,w$ are smooth functions on $(0,1)\times(0,T)$ such that for $0<t<T$ we have $u(\cdot,t)$, $w(\cdot,t)\in D_{\max}$,
\[
u_{tt}+\A u=0,\quad w_{tt}+\A w=0,
\]
and $u(1)=w(1)=0$. Consider the functional
$
E(\tau)
=
\langle u_t(\cdot,\tau),w(\cdot,\tau)\rangle_{\beta}
-
\langle u(\cdot,\tau),w_t(\cdot,\tau)\rangle_{\beta}
$, then
\[
E'(\tau)
=
-\langle Au(\cdot,\tau),w(\cdot,\tau)\rangle_{\beta}
+
\langle u(\cdot,\tau),Aw(\cdot,\tau)\rangle_{\beta}=[u(\cdot,\tau),w(\cdot,\tau)](0),
\]
therefore $E(t)-E(0)=\int_0^t [u(\cdot,\tau),w(\cdot,\tau)](0) \mathrm{d}\tau$.\\

Next we analyze the bracket $[u(\cdot,\tau),w_0](0)$ with $w_0\in D(\A)$.\\
\textit{Assume} $\nu=0$. Let $u(x,t)$ be a sufficiently smooth controlled solution of (\ref{problem}) with a control $f$. Let \(y_\pm\) be the functions given in \eqref{critical}. Motivated by the local asymptotic expansion near the singular endpoint, we formally write
\begin{equation}\label{decomint}
u(x,t)\sim a_u(t)y_+(x)+b_u(t)y_-(x),
\end{equation}
where $
a_u(t)=-[u(\cdot,t),y_-](0)$, $
b_u(t)=[u(\cdot,t),y_+](0)$. We have used the bilinearity of the bracket and $[y_-,y_+](0)=1$, $[y_+,y_+](0)=[y_-,y_-](0)=0.$\\
Let now $w_0\in D(\A)$. By the definition of the self-adjoint realization \(\mathcal A\) we have
$
[w_0,y_+](0)=0$, see (\ref{bouncri}), which implies $w_0(x)\sim-[w_0,y_-](0)y_+(x)$ as $x\to 0^+$, hence
\[
\mathcal O_{-\sigma}(w_0)
=
-[w_0,y_-](0),
\]
and the decomposition (\ref{decomint}) together with the definition of $B_0$ imply that
\[
[u(\cdot,t),w_0](0)
= -[u(\cdot,t), y_+](0)[w_0,y_-](0)=-
(B_0u)(t)[w_0,y_-](0)
=
(B_0u)(t)\mathbb B_0^*
\binom{w_0}{w_1}=f(t)\mathbb B_0^*
\binom{w_0}{w_1}.
\]

\textit{Assume } $\nu>0$. Let $u(x,t)$ be a sufficiently smooth controlled solution of (\ref{problem}) with a control $f$. Let \(\phi_\pm\) be the functions given in \eqref{prin}. Near \(x=0\), we write
\begin{equation}\label{decomayor}
u(x,t)\sim a_u(t)\phi_+(x)+b_u(t)\phi_-(x),
\end{equation}
where $a_u(t)=-[u(\cdot,t), \phi_-](0)$, $b_u(t)=[u(\cdot,t), \phi_+](0)$.\\

\noindent\textit{Subcritical case} $0<\nu<1$. Let $w_0\in D(\A)$ then  $
[w_0,\phi_+](0)=0$, see (\ref{bounori}), which implies $w_0(x)\sim-[w_0,\phi_-](0)\phi_+(x)$ as $x\to 0^+$.\\
\textit{Limit-point case} $\nu\ge1$. In this case the non-principal solution \(\phi_-\notin L^2_\beta(0,1)\). Hence, for sufficiently regular $w_0$ the \(L^2_\beta\)(0,1)-admissible
asymptotic behavior at \(x=0\) is $w_0(x)\sim c\,\phi_+(x)$ as $x\to0^+.$ Since \([\phi_-,\phi_+](0)=1\), we have $w_0(x)\sim-[w_0,\phi_-](0)\phi_+(x)$ as $x\to 0^+$.\\

In both cases, the boundary coefficient associated with the principal asymptotic component satisfies
\[
\mathcal O_{-\sigma-\sqrt\Delta/2}(w_0)
=
-[w_0,\phi_-](0),
\]
and the decomposition (\ref{decomayor}) together with the definition of $B_\nu$ imply that
\begin{eqnarray*}
[u(\cdot,t),w_0](0)&=&-[u(\cdot,t), \phi_+](0)[w_0,\phi_-](0)=-\sqrt{\Delta}(B_\nu u)(t) [w_0,\phi_-](0)\\
			    &=&(B_\nu u)(t)\mathbb B_\nu^*\binom{w_0}{w_1}=f(t)\mathbb B_\nu^*\binom{w_0}{w_1}.
\end{eqnarray*}
These computations suggest the abstract formulation (\ref{ABSTRACT}) and the appearance of the term
\(
f(t)\mathbb B_\nu^*(w_0,w_1)
\)
in the weak formulation below.
\begin{definition}
Let $T>0$, $0\le \alpha<2$, $\beta,\mu\in\mathbb R$ with $\Delta\ge0$.
Given $
U_0=(u_0,u_1)\in X
=
\mathcal H^{\nu+1/2}\times \mathcal H^{\nu-1/2}
$ and $f\in L^2(0,T)$, we say that $U(t)=(u(t),u_t(t))$ is a weak solution of system~\eqref{problem} if $U\in C([0,T];X)$ and, for every $W=(w_0,w_1)\in X_1$, we have
\begin{equation}\label{debilW}
\begin{aligned}
\langle U(t)-U_0,W\rangle_X
&=
\int_0^t
\left\langle U(\tau),\mathbf A^* W\right\rangle_X\,d\tau
+
\int_0^t
f(\tau)\,\mathbb B_\nu^*W\,d\tau
\end{aligned}
\end{equation}
for all $0\le t\le T$.
Equivalently,
\[
\begin{aligned}
&\langle u(t)-u_0,w_0\rangle_{\mathcal H^{\nu+1/2}}
+
\langle u_t(t)-u_1,w_1\rangle_{\mathcal H^{\nu-1/2}}
\\
&=
\int_0^t
\left(
-\langle u(\tau),w_1\rangle_{\mathcal H^{\nu+1/2}}
+
\langle u_t(\tau),\mathcal A w_0\rangle_{\mathcal H^{\nu-1/2}}
\right)d\tau
+
\int_0^t
f(\tau)\,
\mathbb B_\nu^*
\begin{pmatrix}
w_0\\ w_1
\end{pmatrix}
d\tau .
\end{aligned}
\]
\end{definition}

For $\tau>0$, we introduce the input--state mapping $\Phi_\tau :L^2_{\mathrm{loc}}(0,\infty)\longrightarrow X_{-1}$ given by
\[
\Phi_{\tau} f
:=
\int_0^\tau \mathbb T(\tau-s)\,\mathbb B_\nu f(s)\,\mathrm{d}s, \quad f\in L^2_{\mathrm{loc}}(0,\infty),
\]
where the integral is understood in $X_{-1}$, and $\mathbb B_\nu$ is the control operator defined in (\ref{conope}).\\

\begin{remark}
\cite[Theorem 4.4.3]{tucwei} and Remark \ref {admiobco} imply that the dual operator $\mathbb B_\nu\in\mathcal L(\mathbb C,X_{-1})$ is an admissible control operator for $\mathbb{T}(t)$, $t\in \R$.
Furthermore, for every $T>0$, the input map
\[
\Phi_Tf:=\int_0^T\mathbb T(T-\sigma)\mathbb B_\nu f(\sigma)\,d\sigma
\]
defines a bounded operator $\Phi_T\in\mathcal L(L^2(0,T),X)$, see \cite[Proposition 4.2.2]{tucwei}.
\end{remark}

For sake of completeness we include Proposition~4.2.5 in \cite{tucwei}.
\begin{proposition}
Let $T>0$, $0\le \alpha<2$, $\beta,\mu\in\mathbb R$ with $\Delta\ge0$.
For every $U_0\in X$ and every $f\in L^2((0,T);\mathbb C)$, the initial value problem
\begin{equation}\label{ABSTRACT}
\dot U(t)=\mathbf A U(t)+\mathbb B_\nu f(t),
\qquad
U(0)=U_0,
\end{equation}
has a unique solution in $X_{-1}$, i.e.
\[
U(t)-U(0)
=
\int_0^t
\left(\mathbf A U(\sigma)+\mathbb B_\nu f(\sigma)\right)\,d\sigma
\quad\text{in }X_{-1},
\qquad
\forall t\in[0,T].
\]
This solution is given by
\[
U(t)=\mathbb T(t)U_0+\Phi_t f,
\qquad
t\in[0,T],
\]
and it satisfies
\[
U\in C([0,T];X)\cap H^1(0,T;X_{-1}).
\]
In particular, if $u(t)=\pi_1(U(t))$, then
\[
u\in C^0([0,T];\mathcal H^{\nu+1/2})
\cap C^1([0,T];\mathcal H^{\nu-1/2})
\cap H^2(0,T;\mathcal H^{\nu-3/2}).
\]

Moreover, the integral identity in $X_{-1}$ is equivalent to (\ref{debilW})
for all $(w_0,w_1)\in X_1$ and all $0\le t\le T$, see \cite[Remark 4.2.6]{tucwei}.
\end{proposition}

\subsection{Exact observability and controllability}
\begin{definition}
Let $T>0$ be fixed. The pair $(\mathbf A,\mathbb B_\nu)$ is said to be exactly controllable in time $T$ if $
\operatorname{Ran}\Phi_T=X .
$
\end{definition}

\textbf{Proof of Theorem \ref{main}.}
By \cite[Theorem 11.2.1]{tucwei} the pair $(\mathbf A,\mathbb B_\nu)$ is exactly controllable in time $T$ if and only if $(\mathbf A^*=-\mathbf A,\mathbb B_\nu^*)$
is exactly observable in time $T$. By Proposition \ref{observa} and \cite[(6.1.1)]{tucwei} we have that $(-\mathbf A,\mathbb B_\nu^*)$ is exactly observable in any time $T>4/(2-\alpha)$ and the result follows.
\begin{acknowledgement}
	This work was supported by DGAPA-UNAM [PAPIIT IN117525].
	\end{acknowledgement}
		    
\appendix
\section{Bessel functions}
We introduce the Bessel function of the first kind $J_{\nu}$ as follows
\begin{equation}\label{bessel}
J_{\nu}(x)=\sum_{m \geq 0} \frac{(-1)^{m}}{m ! \Gamma(m+\nu+1)}\left(\frac{x}{2}\right)^{2 m+\nu}, \quad x \geq 0,
\end{equation}
where $\Gamma(\cdot)$ is the Gamma function. In particular, for $\nu>-1$ and $0<x \leq \sqrt{\nu+1}$, from (\ref{bessel}) we have (see \cite[9.1.7, p. 360]{abram})
\begin{equation}\label{besasin}
J_{\nu}(x) \sim \frac{1}{\Gamma(\nu+1)}\left(\frac{x}{2}\right)^{\nu} \quad \text { as } \quad x \rightarrow 0^{+} .
\end{equation}
Bessel functions of the first kind satisfy the recurrence formula (see \cite[9.1.27, p. 360]{abram}):
\begin{equation}\label{recur}
x J_{\nu}^{\prime}(x)-\nu J_{\nu}(x)=-x J_{\nu+1}(x).
\end{equation}
For $\nu>0$, and $\nu\notin\mathbb N$, we introduce the Bessel function of the second kind
\[
Y_\nu(z)
:=
\frac{
J_\nu(z)\cos(\nu\pi)-J_{-\nu}(z)
}{
\sin(\nu\pi)
}.
\]
From (\ref{bessel}) we have
\begin{equation}\label{asinu}
Y_\nu(x)
=
-\frac{\Gamma(\nu)}{\pi}\left(\frac{2}{x}\right)^\nu+
\frac{\cot(\pi\nu)}{\Gamma(\nu+1)}
\left(\frac{x}{2}\right)^\nu+o(x^\nu)\quad\text{as } \quad x \rightarrow 0^{+} .
\end{equation}
For $\nu=n\in\mathbb N$, we define
\[
Y_n(x)
:=
\lim_{\nu\to n}
\frac{
J_\nu(x)\cos(\nu\pi)-J_{-\nu}(x)
}{
\sin(\nu\pi)
}.
\]
Therefore,
\begin{equation}\label{asina}
Y_n(x)
=
-\frac{(n-1)!}{\pi}\left(\frac{2}{x}\right)^{n}
+
\frac{2}{\pi}\log\!\left(\frac{x}{2}\right)J_n(x)
+O(x^n)\quad\text {as } \quad x \rightarrow 0^{+} ,
\end{equation}
\begin{equation}\label{asin0}
Y_0(x)=\frac{2}{\pi}\left(\log\frac{x}{2}+\gamma\right)
+O(x^2\log x), \quad\text{as } \quad x \rightarrow 0^{+},
\end{equation}
where $\gamma$ denotes the Euler–Mascheroni constant.\\

Recall the asymptotic behavior of the Bessel function $J_{\nu}$ for large $x$, see \cite[Lem. 7.2, p. 129]{komo}.
\begin{lem}\label{asimxinf}
For any $\nu \in \mathbb{R}$
$$
J_{\nu}(x)=\sqrt{\frac{2}{\pi x}}\left\{\cos \left(x-\frac{\nu \pi}{2}-\frac{\pi}{4}\right)+\mathcal{O}\left(\frac{1}{x}\right)\right\} \quad \text { as } \quad x \rightarrow \infty.
$$
\end{lem}

For $\nu >-1$ the Bessel function $J_{\nu}$ has an infinite number of real zeros $0<j_{\nu, 1}<j_{\nu, 2}<\ldots$, all of which are simple, with the possible exception of $x=0$. In \cite[Proposition 7.8, p. 135]{komo} we can find the next information about the location of the zeros of the Bessel functions $J_{\nu}$:
\begin{lem}\label{consec}Let $\nu \geq 0$.\\
1. The difference sequence $\left(j_{\nu, k+1}-j_{\nu, k}\right)_{k}$ converges to $\pi$ as $k \rightarrow\infty$.\\
2. The sequence $\left(j_{\nu, k+1}-j_{\nu, k}\right)_{k}$ is strictly decreasing if $|\nu|>\frac{1}{2}$, strictly increasing if $|\nu|<\frac{1}{2}$, and constant if $|\nu|=\frac{1}{2}$.\\
\end{lem}

For $\nu \geq 0$ fixed, we consider the next asymptotic expansion of the zeros of the Bessel function $J_{\nu}$, see\cite[Section 15.53, p. 506]{Watson},
\begin{equation}\label{asint}
j_{\nu, k}=\left(k+\frac{\nu}{2}-\frac{1}{4}\right) \pi-\frac{4 \nu^{2}-1}{8\left(k+\frac{\nu}{2}-\frac{1}{4}\right) \pi}+O\left(\frac{1}{k^{3}}\right), \quad \text { as } k \rightarrow\infty
\end{equation}

In particular we have
\begin{equation}\label{below}
\begin{aligned}
&j_{\nu, k} \geq\left(k-\frac{1}{4}\right) \pi \quad \text { for } \nu \in\left[0, 1/2\right], \\
&j_{\nu, k} \geq\left(k-\frac{1}{8}\right) \pi \quad \text { for } \nu \in\left[1/2,\infty\right].
\end{aligned}
\end{equation}

\begin{lem}\label{reduce} For any $\nu \geq 0$ and any $k\geq 1$ we have
$$
\sqrt{j_{\nu, k}}\left|J_{\nu}^{\prime}\left(j_{\nu, k}\right)\right|=\sqrt{\frac{2}{\pi}}+O\left(\frac{1}{j_{\nu, k}}\right)\quad \text{as}\quad k \rightarrow \infty.
$$
\end{lem}
The proof of this result is left to the reader using Lemma \ref{asimxinf} and the recurrence formula (\ref{recur}).\\

\begin{lem}\label{boundtrac}
Let \((\Phi_k)_{k\ge1}\) be the family of eigenfunctions in
\eqref{Phik}. Then the following boundary trace formulas hold.
\begin{enumerate}[label=\rm(\roman*)]
\item If \(\nu=0\), then
\[
\mathcal O_{-\sigma}(\Phi_k)
=
\frac{(2\kappa_\alpha)^{1/2}}
{|J_0'(j_{0,k})|}.
\]
\item If \(\nu>0\), then
\[
\mathcal O_{-\sigma-\sqrt\Delta/2}(\Phi_k)
=
\frac{(2\kappa_\alpha)^{1/2}}
{2^\nu\Gamma(\nu+1)}
\frac{j_{\nu,k}^\nu}
{|J_\nu'(j_{\nu,k})|}.
\]
\end{enumerate}
\end{lem}
\begin{proof}
If \(\nu=0\), then \(J_0(0)=1\), and the formula follows immediately from \eqref{Phik}. 
Assume that \(\nu>0\). Set $
z:=j_{\nu,k}x^{\kappa_\alpha},
$ then
\[
x^{-\sigma-\sqrt\Delta/2}\Phi_k(x)
=
\frac{(2\kappa_\alpha)^{1/2}j_{\nu,k}^{\nu}}
{|J_\nu'(j_{\nu,k})|}
z^{-\nu}J_\nu(z),
\]
and the result follows from (\ref{besasin}).
\end{proof}

The so-called Beurling's theorem can be found in \cite[Theorem 9.2]{komo}.
\begin{teoA}\label{beur}
Let $(\gamma_k)_{k\in\mathbb Z}\subset \mathbb R$ be a sequence satisfying the uniform gap condition
\begin{equation}\label{gap}
\gamma_{k+1}-\gamma_k \ge \delta > 0,
\qquad \forall k\in\mathbb Z.
\end{equation}
Define the upper Beurling density by
\begin{equation}\label{dens}
D^+ := \limsup_{r\to+\infty} \frac{n^+(r)}{r},
\end{equation}
where $k^+(r)$ denotes the maximal number of elements $\gamma_k$
contained in an arbitrary interval of length $r$.\\

Then $\{ e^{i\gamma_k t} \}_{k\in\mathbb Z}$ satisfies the upper Ingham-type inequality for every bounded interval \(I\).
If, in addition, \(|I|>2\pi D^+\), then the lower inequality also holds 
\begin{equation*}
c_1 \sum |a_k|^2
\le
\int_I \left| \sum a_k e^{i\gamma_k t} \right|^2 dt
\le
c_2 \sum |a_k|^2
\end{equation*}
for some constants $c_1,c_2>0$ and for every finitely supported sequence
$(a_n)\subset \mathbb C$.
If $|I| < 2\pi D^+$, the above inequality fails.
\end{teoA} 

\begin{defA}[The Friedrichs Extension]\label{fried}
Suppose $S_{\min}:D_{\min}\subset L^2_\beta(0,1)\rightarrow L^2_\beta(0,1)$ is bounded below. Let $D(S_F)$ denote the set of all $u\in D_{\max}$ for which there exists a sequence $\{u_n\in D_{\min}: n\in\mathbb{N}\}$ such that
\begin{enumerate}
\item $u_n \to u$ in $L^2_\beta(0,1)$ as $n\to\infty$,
\item $\langle S_{\min}(u_n-u_m),\,u_n-u_m\rangle_{L^2_\beta(0,1)}\to 0$ as $n,m\to\infty$.
\end{enumerate}
Define the operator $S_F$ by
\[
S_Fu= S_{\max}u, \qquad u\in D(S_F).
\]
Then $S_F$ is called the Friedrichs extension of $S_{\min}$. According to the well known result of Friedrichs \cite{Fried}, $S_F$ is a self-adjoint operator with the same lower bound as $S_{\min}$.
\end{defA}

\end{document}